\documentclass[12pt,fleqn]{article}

\usepackage[cp1251]{inputenc}
\usepackage[T2A]{fontenc}
\usepackage[english,russian]{babel}

\usepackage{amsmath}
\usepackage{amssymb}
\usepackage{amsthm}
\usepackage{array}
\usepackage[auth-sc]{authblk}
\usepackage{bm}
\usepackage{cite}
\usepackage{color}
\usepackage{dsfont}
\usepackage{enumerate}
\usepackage{enumitem}
\usepackage{epsf}
\usepackage{euscript}
\usepackage{graphicx} 
\usepackage{indentfirst}
\usepackage{latexsym}
\usepackage{mathrsfs}
\usepackage{mathtools}
\usepackage{tikz}
\usepackage{pgf}

\newtheorem{theorem}{\indent\bf Теорема}[section]

\newtheorem{lemma}[theorem]{\indent\bf Лемма}

\newtheorem{proposition}[theorem]{\indent\bf Предложение}

\newlength{\templengths}

\definecolor{mathcolor}{RGB}{0,0,0} 
\definecolor{argcolor} {RGB}{0,0,0} 
\definecolor{speccolor}{RGB}{0,0,0} 
\newcommand{\dismath} [1] {{\color{mathcolor}#1}}
\newcommand{\distext} [1] {{\color{speccolor}#1}}
\newcommand{\argument}[1] {{\color{argcolor!50!black}{#1}}}

\SetEnumitemKey{midsep}{itemsep=0pt}
\setlist{midsep}

\newcommand{\arrayitem}{\distext{\!\!\!\!\!\!\!\!\bullet}}
\newcommand{\arrayitemskip}     {\medskip}

\newcommand{\defnotion}       [1]       {\distext{\textit{#1\/}}}

\newcommand{\remember}        [1]       {}

\newcommand{\prop}{{\lang{P}}} 
\newcommand{\uclosure}        [1]       {\dismath{\bar{\forall}}#1} 

\newcommand{\set}             [1]       {\dismath{\left\{ \argument{#1} \right\}}}
\newcommand{\setc}            [2]       {\set{ \argument{#1} \mathrel{\dismath{:}} \argument{#2} }}

\newcommand{\otuple}          [1]       {\dismath{\left\langle\argument{#1}\right\rangle}}
\newcommand{\opair}           [2]       {\otuple{{#1},{#2}}}

\newcommand{\clmodels}        [0]   
                                    {\mathrel{\dismath{|\hspace{-1.25pt}{\models}}}}

\newcommand{\nameKFrame}      [1]       {\dismath{\mathfrak{#1}}}       
\newcommand{\nameKModel}      [1]       {\dismath{\mathfrak{#1}}}       
\newcommand{\nameClModel}     [1]       {\dismath{{#1}}}        
\newcommand{\kframe}          [1]       {\nameKFrame{#1}}       
\newcommand{\kmodel}          [1]       {\nameKModel{#1}}       
\newcommand{\cmodel}          [1]       {\nameClModel{#1}}      

\newcommand{\logic}           [1]       {\dismath{\mathbf{#1}}}         
\newcommand{\lang}            [1]       {\dismath{\mathcal{#1}}}        
\newcommand{\cclass}          [1]       {\dismath{\mathrm{#1}}}         
  
\newcommand{\numbers}         [1]       {\dismath{\mathds{#1}}}

\newcommand{\numNp}                     {\numbers{N}^+}

\newcommand{\otupleIs}        [2]       {\ensuremath \argument{#1} = \otuple{#2}}
\newcommand{\kframeIs}        [2]       {\otupleIs{\nameKFrame{#1}}{#2}}

\newcommand{\kfmodelIs}       [3]       {\otupleIs{\nameKModel{#1}}{\nameKFrame{#2},{#3}}}

\newcommand{\implication}               {\to}
  \newcommand{\imp} {\implication}
\newcommand{\conjunction}               {\wedge}   
  \newcommand{\con} {\conjunction}

\renewcommand{\iff}                     {\mathrel{\dismath{\Longleftrightarrow}}} 
\newcommand{\imply}                     {\mathrel{\dismath{\Longrightarrow}}}     
\newcommand{\bydef}                     {\mathrel{\dismath{\leftrightharpoons}}}  

\makeatletter
\newcommand{\mref}{\@ifnextchar({\mref@i}{\mref@i({\Box},{p})}}
\def\mref@i(#1,#2){#1#2 \implication #2}
\makeatother
    \makeatletter
    \newcommand{\mrefp}{\@ifnextchar({\mrefp@i}{\mrefp@i({p})}}
    \def\mrefp@i(#1){\mref(\Box,#1)}
    \makeatother
\makeatletter
\newcommand{\FOref}{\@ifnextchar({\FOref@i}{\FOref@i({x},{P})}}
\def\FOref@i(#1,#2){\forall #1\,#2(#1,#1)}
\makeatother
    \makeatletter
    \newcommand{\FOrefp}{\@ifnextchar({\FOrefp@i}{\FOrefp@i({P})}}
    \def\FOrefp@i(#1){\FOref(x,#1)}
    \makeatother
\makeatletter
\newcommand{\FOrefi}{\@ifnextchar({\FOrefi@i}{\FOrefi@i({x},{P})}}
\def\FOrefi@i(#1,#2){\forall #1\,#1#2#1}
\makeatother
    \makeatletter
    \newcommand{\FOrefip}{\@ifnextchar({\FOrefip@i}{\FOrefip@i({P})}}
    \def\FOrefip@i(#1){\FOrefi(x,#1)}
    \makeatother

\makeatletter
\newcommand{\mtra}{\@ifnextchar({\mtra@i}{\mtra@i({\Box},{p})}}
\def\mtra@i(#1,#2){#1#2 \implication #1#1#2}
\makeatother
    \makeatletter
    \newcommand{\mtrap}{\@ifnextchar({\mtrap@i}{\mtrap@i({p})}}
    \def\mtrap@i(#1){\mtra(\Box,#1)}
    \makeatother
\makeatletter
\newcommand{\FOtra}{\@ifnextchar({\FOtra@i}{\FOtra@i({x},{y},{z},{P})}}
\def\FOtra@i(#1,#2,#3,#4){\forall #1\forall #2\forall #3\,(#4(#1,#2)\conjunction #4(#2,#3) \implication #4(#1,#3))}
\makeatother
    \makeatletter
    \newcommand{\FOtrap}{\@ifnextchar({\FOtrap@i}{\FOtrap@i({P})}}
    \def\FOtrap@i(#1){\FOtra(x,y,z,#1)}
    \makeatother
\makeatletter
\newcommand{\FOtrai}{\@ifnextchar({\FOtrai@i}{\FOtrai@i({x},{y},{z},{P})}}
\def\FOtrai@i(#1,#2,#3,#4){\forall #1\forall #2\forall #3\,(#1#4#2\conjunction #2#4#3 \implication #1#4#3)}
\makeatother
    \makeatletter
    \newcommand{\FOtraip}{\@ifnextchar({\FOtraip@i}{\FOtraip@i({P})}}
    \def\FOtraip@i(#1){\FOtrai(x,y,z,#1)}
    \makeatother

\makeatletter
\newcommand{\msym}{\@ifnextchar({\msym@i}{\msym@i({\Box},{\Diamond},{p})}}
\def\msym@i(#1,#2,#3){#3 \implication #1#2#3}
\makeatother
    \makeatletter
    \newcommand{\msymp}{\@ifnextchar({\msymp@i}{\msymp@i({p})}}
    \def\msymp@i(#1){\msym(\Box,\Diamond,#1)}
    \makeatother
\makeatletter
\newcommand{\FOsym}{\@ifnextchar({\FOsym@i}{\FOsym@i({x},{y},{P})}}
\def\FOsym@i(#1,#2,#3){\forall #1\forall #2\,(#3(#1,#2)\implication #3(#2,#1))}
\makeatother
    \makeatletter
    \newcommand{\FOsymp}{\@ifnextchar({\FOsymp@i}{\FOsymp@i({P})}}
    \def\FOsymp@i(#1){\FOsym(x,y,#1)}
    \makeatother
\makeatletter
\newcommand{\FOsymi}{\@ifnextchar({\FOsymi@i}{\FOsymi@i({x},{y},{P})}}
\def\FOsymi@i(#1,#2,#3){\forall #1\forall #2\,(#1#3#2 \implication #2#3#1)}
\makeatother
    \makeatletter
    \newcommand{\FOsymip}{\@ifnextchar({\FOsymip@i}{\FOsymip@i({P})}}
    \def\FOsymip@i(#1){\FOsymi(x,y,#1)}
    \makeatother

\makeatletter
\newcommand{\meuc}{\@ifnextchar({\meuc@i}{\meuc@i({\Box},{\Diamond},{p})}}
\def\meuc@i(#1,#2,#3){#2#3 \implication #1#2#3}
\makeatother
    \makeatletter
    \newcommand{\meucp}{\@ifnextchar({\meucp@i}{\meucp@i({p})}}
    \def\meucp@i(#1){\meuc(\Box,\Diamond,#1)}
    \makeatother
\makeatletter
\newcommand{\FOeuc}{\@ifnextchar({\FOeuc@i}{\FOeuc@i({x},{y},{z},{P})}}
\def\FOeuc@i(#1,#2,#3,#4){\forall #1\forall #2\forall #3\,(#4(#1,#2)\con #4(#1,#3)\implication #4(#2,#3))}
\makeatother
    \makeatletter
    \newcommand{\FOeucp}{\@ifnextchar({\FOeucp@i}{\FOeucp@i({P})}}
    \def\FOeucp@i(#1){\FOeuc(x,y,z,#1)}
    \makeatother
\makeatletter
\newcommand{\FOeuci}{\@ifnextchar({\FOeuci@i}{\FOeuci@i({x},{y},{z},{P})}}
\def\FOeuci@i(#1,#2,#3,#4){\forall #1\forall #2\forall #3\,(#1#4#2 \con #1#4#3 \implication #2#4#3)}
\makeatother
    \makeatletter
    \newcommand{\FOeucip}{\@ifnextchar({\FOeucip@i}{\FOeucip@i({P})}}
    \def\FOeucip@i(#1){\FOeuci(x,y,z,#1)}
    \makeatother

\makeatletter
\newcommand{\mser}{\@ifnextchar({\mser@i}{\mser@i({\Box},{\Diamond},{p})}}
\def\mser@i(#1,#2,#3){#1#3 \implication #2#3}
\makeatother
    \makeatletter
    \newcommand{\mserp}{\@ifnextchar({\mserp@i}{\mserp@i({p})}}
    \def\mserp@i(#1){\mser(\Box,\Diamond,#1)}
    \makeatother
\makeatletter
\newcommand{\FOser}{\@ifnextchar({\FOser@i}{\FOser@i({x},{y},{P})}}
\def\FOser@i(#1,#2,#3){\forall #1\exists #2\,#3(#1,#2)}
\makeatother
    \makeatletter
    \newcommand{\FOserp}{\@ifnextchar({\FOserp@i}{\FOserp@i({P})}}
    \def\FOserp@i(#1){\FOser(x,y,#1)}
    \makeatother
\makeatletter
\newcommand{\FOseri}{\@ifnextchar({\FOseri@i}{\FOseri@i({x},{y},{P})}}
\def\FOseri@i(#1,#2,#3){\forall #1\exists #2\,#1#3#2}
\makeatother
    \makeatletter
    \newcommand{\FOserip}{\@ifnextchar({\FOserip@i}{\FOserip@i({P})}}
    \def\FOserip@i(#1){\FOseri(x,y,#1)}
    \makeatother

\makeatletter
\newcommand{\mla}{\@ifnextchar({\mla@i}{\mla@i({\Box},{p})}}
\def\mla@i(#1,#2){#1(#1#2 \implication #2) \implication #1#2}
\makeatother
    \makeatletter
    \newcommand{\mlap}{\@ifnextchar({\mlap@i}{\mlap@i({p})}}
    \def\mlap@i(#1){\mla(\Box,#1)}
    \makeatother

\makeatletter
\newcommand{\mgrz}{\@ifnextchar({\mgrz@i}{\mgrz@i({\Box},{p})}}
\def\mgrz@i(#1,#2){#1(#1(#2 \implication #1#2) \implication #2) \implication #2}
\makeatother
    \makeatletter
    \newcommand{\mgrzp}{\@ifnextchar({\mgrzp@i}{\mgrzp@i({p})}}
    \def\mgrzp@i(#1){\mgrz(\Box,#1)}
    \makeatother

\makeatletter
\newcommand{\mwgrz}{\@ifnextchar({\mwgrz@i}{\mwgrz@i({\Box},{p})}}
\def\mwgrz@i(#1,#2){#1^+(#1(#2 \implication #1#2) \implication #2) \implication #2}
\makeatother
    \makeatletter
    \newcommand{\mwgrzp}{\@ifnextchar({\mwgrzp@i}{\mwgrzp@i({p})}}
    \def\mwgrzp@i(#1){\mwgrz(\Box,#1)}
    \makeatother

\begin{document}

\title{Сложность константного фрагмента \\ слабой логики Гжегорчика\thanks{Работа И.\,А.~Агаджанян и М.\,Н.~Рыбакова поддержана программой <<Научный фонд НИУ ВШЭ>>, грант \mbox{21-04-027}.}}

\author{И.\,А.~Агаджанян$^1$, М.\,Н.~Рыбаков$^2$}
\affil{$^1$НИУ ВШЭ \\ $^2$ИППИ РАН, НИУ ВШЭ}

\date{}

\maketitle

    \section{Сложность логик в полном языке}


        Начнём с того, что приведём краткое обоснование $\cclass{PSPACE}$-трудности проблемы разрешения логик (и вообще любых множеств формул), заключённых между $\logic{K}$ и $\logic{GL}$, между $\logic{K}$ и $\logic{Grz}$, а также между $\logic{K}$ и $\logic{KTB}$; нам важны будут возникающие модальные формулы. Для этого покажем, как, используя конструкцию Р.\,Ладнера~\cite{Ladner77}, полиномиально свести проблему истинности булевых формул с кванторами к произвольной логике указанного класса.

        Считаем, что булевы формулы с кванторами строятся из пропозициональных переменных множества $\prop$ и константы~$\bot$ с помощью связок $\wedge$, $\vee$, $\to$ и кванторов по пропозициональным переменным $\forall p$ и $\exists p$, где $p\in\prop$. Пусть $\lang{QBF}$~--- множество всех булевых формул с кванторами.

        В формулах вида $\forall p\,\varphi$ и $\exists p\,\varphi$ формула $\varphi$ называется областью действия квантора $\forall p$ или $\exists p$, соответственно. Вхождение переменной $p$ называется замкнутым, если оно находится в области действия квантора по этой переменной; иначе оно называется свободным. Свободными переменными формулы называются переменные, имеющие свободное вхождение в эту формулу. Если формула не имеет свободных переменных, то она называется замкнутой.

        Универсальным замыканием булевой формулы с кванторами $\varphi$ называем формулу $\forall q_1\ldots\forall q_n\,\varphi$, где $q_1,\ldots,q_n$~--- все свободные переменные формулы~$\varphi$. Универсальное замыкание формулы $\varphi$ определено однозначно с точностью до порядка следования добавляемых перед $\varphi$ кванторов всеобщности; будем считать, что этот порядок соответствует порядку индексов свободных переменных и
        для универсального замыкания формулы $\varphi$ будем использовать обозначение~$\uclosure{\varphi}$.

        Истинность булевых формул с кванторами определяется обычно. Именно, пусть $\cmodel{M}\subseteq\prop$ и $\varphi$~--- булева формула с кванторами; множество $\cmodel{M}$ называем моделью. Отношение $\cmodel{M}\models\varphi$ задаётся согласно следующему рекурсивному определению:
        \[
          \begin{array}{clcl}
          \arrayitem & \mbox{$\cmodel{M}\models p_i$}
                     & \leftrightharpoons
                     & p_i\in \cmodel{M};
                     \arrayitemskip\\
          \arrayitem & \mbox{$\cmodel{M}\not\models \bot$;}
                     \arrayitemskip\\
          \arrayitem & \mbox{$\cmodel{M}\models \varphi' \wedge \varphi''$}
                     & \leftrightharpoons
                     & \mbox{$\cmodel{M} \models \varphi'$ $\phantom{\mbox{л}}$и$\phantom{\mbox{и}}$ $\cmodel{M}\models \varphi''$;}
                     \arrayitemskip\\
          \arrayitem & \mbox{$\cmodel{M}\models \varphi' \vee \varphi''$}
                     & \leftrightharpoons
                     & \mbox{$\cmodel{M} \models \varphi'$ или $\cmodel{M}\models \varphi''$;}
                     \arrayitemskip\\
          \arrayitem & \mbox{$\cmodel{M}\models \varphi' \to \varphi''$}
                     & \leftrightharpoons
                     & \mbox{$\cmodel{M} \not\models \varphi'$ или $\cmodel{M}\models \varphi''$;}
                     \arrayitemskip\\
          \arrayitem & \mbox{$\cmodel{M}\models \forall p_i\,\varphi'$}
                     & \leftrightharpoons
                     & \mbox{$\cmodel{M}\cup\{p_i\} \models \varphi'$ $\phantom{\mbox{л}}$и$\phantom{\mbox{и}}$ $\cmodel{M}\setminus\{p_i\}\models \varphi'$;}
                     \arrayitemskip\\
          \arrayitem & \mbox{$\cmodel{M}\models \exists p_i\,\varphi'$}
                     & \leftrightharpoons
                     & \mbox{$\cmodel{M}\cup\{p_i\} \models \varphi'$ или $\cmodel{M}\setminus\{p_i\}\models \varphi'$.}
        \end{array}
        \]
        Булеву формулу с кванторами называем тождественно истинной, если она истинна в любой модели.

        Заметим, что тождественная истинность формулы $\varphi$ равносильна тому, что её универсальное замыкание истинно в некоторой модели, а также тому, что её универсальное замыкание истинно в любой модели. Пусть
        $$
        \begin{array}{lcl}
        \logic{TQBF} & = & \{\varphi\in\lang{QBF} : \varnothing\models\uclosure{\varphi}\},
        \end{array}
        $$
        т.\,е. $\logic{TQBF}$~--- множество всех тождественно истинных булевых формул с кванторами.

        Будем говорить, что булева формула с кванторами $\varphi$ находится в {\defnotion{префиксной форме},\index{уяа@форма!префиксная} если $\varphi=Q_1q_1\ldots Q_nq_n\,\varphi'$, где $Q_1,\ldots,Q_n\in \{\forall,\exists\}$, $q_1,\ldots,q_n\in\prop$, а $\varphi'$~--- бескванторная формула.

        Известно, что задачи принадлежности формул множеству $\logic{TQBF}$ является $\cclass{PSPACE}$\nobreakdash-полной, см.~\cite{Stockmeyer-1987-1}; причём достаточно ограничиться рассмотрением замкнутых формул в префиксной форме. Используя этот факт, Р.\,Ладнер доказал $\cclass{PSPACE}$\nobreakdash-трудность логик $\logic{K}$, $\logic{T}$ и $\logic{S4}$~\cite{Ladner77}, а позднее его конструкция была обобщена и на другие модальные логики, см., например,~\cite{ZakharyaschevWolterChagrov-2001}.

        Для наших целей понадобится незначительная модификация конструкции Р.\,Ладнера; опишем её.

        Пусть $\varphi = Q_1p_1\ldots Q_np_n\,\varphi'$~--- замкнутая булева формула в префиксной форме. Пусть $q_0,\ldots,q_{n+1}$~--- пропозициональные переменные, не входящие в $\varphi$; содержательный смысл переменной $q_i$~--- в том, что мы <<раскрыли>> не менее чем $i$~кванторов в формуле~$\varphi$. Пусть $\varphi^\ast$~---
        \label{varphi_ast}
        \label{varphi:ast}
        конъюнкция следующих формул:
          \[
          \begin{array}{ll}
          \arrayitem
            & \displaystyle q_0\wedge\bigwedge\limits_{i=1}^n \neg p_i; \smallskip\\
          \arrayitem
            & \displaystyle
            \Box^{\leqslant n} \bigwedge\limits_{i = 1}^n ( q_i \imp q_{i-1}); \smallskip\\
          \arrayitem
            & \displaystyle
            \Box^{\leqslant n-1} \bigwedge\limits_{{Q}_i =
              \exists} ( q_{i-1}\wedge\neg q_{i} \imp \Diamond (q_{i}\wedge\neg q_{i+1}); \smallskip\\
          \arrayitem
            & \displaystyle
            \Box^{\leqslant n-1} \bigwedge\limits_{{Q}_i = \forall} (q_{i-1}\wedge\neg q_{i} \imp \Diamond (q_{i}\wedge\neg q_{i+1}
            \con p_{i}) \con \Diamond (q_{i}\wedge\neg q_{i+1} \con \neg p_{i})); \smallskip\\
          \arrayitem
            & \displaystyle
            \Box^{\leqslant n-1} \bigwedge\limits_{i = 1}^{n-1} ( q_i \imp
            \bigwedge\limits_{j \leqslant i} (\phantom{\neg}p_j \imp \Box (q_{i+1}\wedge\neg p_{n+1} \imp \phantom{\neg}p_j)) \con {} \\
            & \displaystyle \phantom{\Box^{\leqslant n-1} \bigwedge\limits_{i = 1}^{n-1} ( q_i \imp\,}
            \bigwedge\limits_{j \leqslant i} (\neg p_j \imp \Box
            (q_{i+1}\wedge\neg p_{n+1} \imp \neg p_j))); \smallskip\\
          \arrayitem
            & \displaystyle \Box^{n} (q_n\wedge\neg q_{n+1} \imp \varphi').
          \end{array}
          \]

            Формула $\varphi^\ast$ <<объясняет>>, как последовательно <<раскрываются>> кванторы формулы~$\varphi$, а также требует, чтобы после <<раскрытия>> кванторов формула~$\varphi'$ была истинной. Более точно, выполняется следующее утверждение.

            \begin{proposition}
            \label{prop:modal:complexity}
            Пусть $\logic{K}\subseteq L\subseteq\logic{GL}$
            или   $\logic{K}\subseteq L\subseteq\logic{Grz}$
            или   $\logic{K}\subseteq L\subseteq\logic{KTB}$.
            Тогда справедлива следующая эквивалентность:
            $$
            \begin{array}{lcl}
            \neg\varphi\in\logic{TQBF} & \iff & \neg\varphi^\ast\in L.
            \end{array}
            $$
            \end{proposition}

            \begin{proof}
            Покажем, что
            \begin{itemize}
                \item
                если $\neg\varphi\not\in\logic{TQBF}$, то $\neg\varphi^\ast\not\in\logic{GL}$, $\neg\varphi^\ast\not\in\logic{Grz}$ и $\neg\varphi^\ast\not\in\logic{KTB}$;
                \item
                если $\neg\varphi^\ast\not\in\logic{K}$, то $\neg\varphi\not\in\logic{TQBF}$,
            \end{itemize}
            откуда и будет следовать указанная эквивалентность.

            Пусть $\neg\varphi\not\in\logic{TQBF}$. Тогда $\varphi$~--- истинная формула. Это означает, что мы можем построить кванторное дерево для~$\varphi$; опишем его построение. В корне этого дерева находится вершина, являющаяся классической моделью $\cmodel{M}_\varphi^0=\varnothing$. Пусть уже построена вершина $\cmodel{M}^s_\psi$ для формулы $\psi=Q_kp_k\ldots Q_np_n\,\varphi'$, где $k\in\{1,\ldots,n\}$.
            Если $Q_k=\forall$, то добавляем вершины $\cmodel{M}^0_{\psi'} = \cmodel{M}^s_\psi$ и $\cmodel{M}^1_{\psi'} = \cmodel{M}^s_\psi \cup \{p_k\}$, где $\psi'$~--- формула, полученная из~$\psi$ удалением квантора $Q_kp_k$.
            Если же $Q_k=\exists$, то поступаем следующим образом: добавляем какую-нибудь из вершин $\cmodel{M}^0_{\psi'}$ и~$\cmodel{M}^1_{\psi'}$, в которой истинна формула~$\psi'$.
            В каждую из добавленных вершин $\cmodel{M}^0_{\psi'}$ и~$\cmodel{M}^1_{\psi'}$ проводим дугу из вершины~$\cmodel{M}^s_\psi$. Построение кванторного дерева заканчивается, когда не остаётся висячих вершин, помеченных формулами, содержащими хотя бы один квантор.

            Заметим, что если вершина $\cmodel{M}_\psi^s$ оказалась в кванторном дереве для~$\varphi$, то $\cmodel{M}_\psi\models\psi$.

            Определим шкалу $\otupleIs{\kframe{F}}{W,R}$, взяв в качестве $W$ множество всех вершин кванторного дерева для~$\varphi$, а в качестве~$R$~--- множество всех дуг в этом дереве. Определим оценку~$v$:
            $$
            \begin{array}{lcl}
            v(p_k) & = & \{w\in W : p_k\in w\}.
            \end{array}
            $$
            Пусть теперь $R^{t}$, $R^{rt}$ и $R^{rs}$~--- транзитивное, рефлексивно-транзитивное и рефлексивно-симметричное замыкания отношения~$R$, соответственно. Несложно понять, что в этом случае шкалы
            $\otupleIs{\kframe{F}^{t}}{W,R^{t}}$, $\otupleIs{\kframe{F}^{rt}}{W,R^{rt}}$ и $\otupleIs{\kframe{F}^{rs}}{W,R^{rs}}$ являются шкалами логик $\logic{GL}$, $\logic{Grz}$ и $\logic{KTB}$, соответственно, причём в корне~$\cmodel{M}_\varphi^0$ каждой из этих шкал при оценке~$v$ истинна формула~$\varphi^\ast$. Следовательно, $\neg\varphi^\ast\not\in\logic{GL}$, $\neg\varphi^\ast\not\in\logic{Grz}$ и $\neg\varphi^\ast\not\in\logic{KTB}$.

            Пусть $\neg\varphi^\ast\not\in\logic{K}$. Тогда существует модель Крипке, в некотором мире которой истинна формула~$\varphi^\ast$. Но истинность формулы~$\varphi^\ast$ позволяет <<шаг за шагом>> извлечь из этой модели кванторное дерево для формулы~$\varphi$, подтверждающее её истинность, а значит, $\neg\varphi\not\in\logic{TQBF}$.
            \end{proof}

            \label{observation:modal:complexity:heredity}
            Заметим, что в предложении~\ref{prop:modal:complexity} оценка $v$, определённая на шкалах $\otupleIs{\kframe{F}^{t}}{W,R^{t}}$, $\otupleIs{\kframe{F}^{rt}}{W,R^{rt}}$, удовлетворяет следующему условию, которое назовём условием наследственности <<вверх>>:
            $$
            \begin{array}{lcl}
            \mbox{$w\in v(p_k)$ и $wRw'$} & \imply & w'\in v(p_k).
            \end{array}
            $$
            Это наблюдение станет важным, когда мы будем устанавливать нижние границы сложности для фрагментов логик, полных относительно классов шкал, в которых одним из необходимых свойств отношения достижимости является транзитивность.

\section{Интервал $[\logic{K},\logic{wGrz}]$}
    \label{subsection:K-K4}

Логика $\logic{K}$ определяется как множество модальных формул, истинных в классе всех шкал Крипке (см.~\cite{ChZ}). Слабая логика Гжегорчика $\logic{wGrz}$ (см.~\cite{Litak}) определяется как исчисление над $\logic{K}$ следующим образом: $\logic{wGrz} = \logic{K}\oplus \Box^+(\Box(p\to\Box p)\to p)\to p$, где $\oplus$ означает замыкание по modus ponens, правилу Гёделя и правилу подстановки, а $\Box^+\varphi = \varphi\wedge \Box\varphi$.

Будем использовать идеи Дж.\,Халперна~\cite{Halpern95}, а также их модификации, см.~\cite{RybakovChagrov-2002-2-rus, ChRyb03, Rybakov-2003-2-rus, Rybakov-2004-1-rus, Rybakov06, Rybakov-2007-1-rus, Rybakov07, Rybakov08, RShICTAC18, RShIGPL18, RShSaicsit18, RShIGPL19, RShJLC21a, ARSh2021, RShJLC22, RShTCS22}.

            Пусть $\varphi$~--- булева формула с кванторами, $p_1,\dots,p_n$~---
            все переменные, входящие в~$\varphi$. Пусть $\varphi^\ast$~---
            формула, построенная по $\varphi$ так, как это описано на
            стр.\,\pageref{varphi:ast}.  Напомним, что переменными формулы
            $\varphi^\ast$ являются $p_1,\dots,p_{n},q_0,\ldots,q_{n+1}$. Для единообразия обозначений будем считать, что $q_0=p_{n+1},\ldots,q_{n+1}=p_{2n+2}$, т.\,е. переменными формулы $\varphi^\ast$ являются $p_1,\dots,p_{2n+2}$.

            Для всякого $k\in\numNp$ положим
            $$
            \begin{array}{lcl}
            \alpha_k
            & = &
            \Box(\Diamond^k\Box\bot \wedge \neg
            \Diamond^{k+1}\Box\bot \to \Box(\Diamond \top \to \Diamond\Box\bot))
            \end{array}
            $$
            и обозначим через $\varphi_\alpha^\ast$ формулу, получающуюся из
            $\varphi^\ast$ подстановкой $\alpha_1,\dots,\alpha_{2n+2}$ вместо
            $p_1,\dots,p_{2n+2}$ соответственно.

            \begin{lemma}
            \label{lem_K4(0)form_equiv}
            \label{lem:K4(0)form:equiv}
            Формула $\varphi_\alpha^\ast$ строится по $\varphi^\ast$
            некоторым алгоритмом за полиномиальное время от $|\varphi^\ast|$, при
            этом
            $$
            \begin{array}{rcl}
            \mbox{$\varphi_\alpha^\ast$ $\logic{wGrz}$-выполнима}
            & \iff &
            \mbox{$\varphi^\ast$ $\logic{wGrz}$-выполнима.}
            \end{array}
            $$
            \end{lemma}

            \begin{proof}
            Заметим, что формулы  $\alpha_1,\dots,\alpha_{2n+2}$ строятся по
            $\varphi^\ast$ полиномиально, поскольку для
            некоторой константы~$c$
            $$
            |\alpha_m|
              ~\leqslant~ c\cdot m
              ~\leqslant~ c\cdot(2n + 2)
              ~\leqslant~ c\cdot|\varphi^\ast|.
            $$
            Следовательно,
            $$
            |\varphi_\alpha^\ast|
              ~\leqslant~ \max\{|\alpha_1|,\dots,|\alpha_{2n+2}|\}\cdot|\varphi^\ast|
              ~\leqslant~ c\cdot|\varphi^\ast|^2,
            $$
            из чего несложно заключить, что для построения формулы
            $\varphi_\alpha^\ast$ по $\varphi^\ast$ достаточно
            полиномиального времени от $|\varphi^\ast|$.

            Пусть $\varphi^\ast$ не является $\logic{wGrz}$-выполнимой. Тогда
            $\neg\varphi^\ast\in\logic{wGrz}$. Поскольку $\neg\varphi^\ast_\alpha$
            получена подстановкой из $\neg\varphi^\ast$, то
            $\neg\varphi^\ast_\alpha\in\logic{wGrz}$, и следовательно,
            $\varphi^\ast_\alpha$ не является $\logic{K4}$\nobreakdash-выполнимой.

            Пусть теперь $\varphi^\ast$~--- $\logic{wGrz}$-выполнимая формула. Это
            означает, что булева формула с кванторами $\varphi$ истинна.
            Следовательно, $\varphi^\ast$ истинна в некотором мире $w_0$
            модели $\kfmodelIs{M}{F}{v}$, определённой на
            $\logic{Grz}$-шкале $\kframeIs{F}{W,R}$ высоты $n\,{+}\,1$ (см.~доказательство
            предложения~\ref{prop:modal:complexity}). Несложно заметить, что оценка
            $v$ является наследственной: для всякой переменной $p_i$, где $0\leqslant
            i\leqslant 2n\,{+}\,2$, и для всяких $w',w''\in W$ таких, что
            $w'Rw''$ и $(\kmodel{M}, w') \models p_i$, выполнено
            $(\kmodel{M}, w'') \models p_i$.

            По модели $\kmodel{M}$ построим $\logic{wGrz}$-шкалу $\kframe{F}'=\opair{W'}{R'}$ и модель $\kmodel{M}' = \opair{\mathfrak{F}'}{v'}$, определённую на этой шкале, в некотором мире которой истинна формула~$\varphi^\ast_\alpha$.

\begin{figure}
  \centering
  \begin{tikzpicture}[scale=1.50]

    \coordinate (a30)    at ( 0.0, 3);
    \coordinate (a31)    at ( 0.0, 4);
    \coordinate (a32)    at ( 0.0, 5);
    \coordinate (a34)    at ( 0.0, 6);
    \coordinate (a35)    at ( 0.0, 7);
    \coordinate (b30)    at (-1.0, 4);
    \coordinate (c30)    at ( 0.0, 2);

    \coordinate (ld)     at (-3.2, 2.5);
    \coordinate (lu)     at (-3.2, 7.5);
    \coordinate (ru)     at ( 1.0, 7.5);
    \coordinate (rd)     at ( 1.0, 2.5);

    \draw [fill]     (a30)   circle [radius=2.0pt] ;
    \draw [fill]     (a31)   circle [radius=2.0pt] ;
    \draw [fill]     (a32)   circle [radius=2.0pt] ;
    \draw [fill]     (a35)   circle [radius=2.0pt] ;
    \draw []     (b30)   circle [radius=2.0pt] ;
    \draw []     (c30)   circle [radius=2.0pt] ;

    \begin{scope}[>=latex]

      \draw [->, shorten >=  2.75pt, shorten <= 2.75pt] (c30) -- (a30) ;
      \draw [->, shorten >=  2.75pt, shorten <= 2.75pt] (a30) -- (a31) ;
      \draw [->, shorten >=  2.75pt, shorten <= 2.75pt] (a31) -- (a32) ;
      \draw [->, shorten >= 11.75pt, shorten <= 2.75pt] (a32) -- (a34) ;
      \draw [->, shorten >=  2.75pt, shorten <= 7.75pt] (a34) -- (a35) ;
      \draw [->, shorten >=  2.75pt, shorten <= 2.75pt] (a30) -- (b30) ;

    \end{scope}

      \node []         at (a34)  {$\vdots$};
      \node [left=2pt] at (a35)  {$\Box\bot$};

      \node [right=2pt] at (a30)  {$a^m_0$};
      \node [right=2pt] at (a31)  {$a^m_1$};
      \node [right=2pt] at (a32)  {$a^m_2$};
      \node [right=2pt] at (a35)  {$a^m_{m}$};

      \node [left=2pt]  at (b30)  {$\Diamond\top\wedge\Box\Diamond\top$};
      \node [above=2pt] at (b30)  {$~~b^m$};
      \node [left=2pt]  at (c30)  {$\phantom{c^m}\neg\alpha_m$};
      \node [right=2pt] at (c30)  {$c^m$};

      \draw [dashed, color=black!50!] (ld) -- (lu) -- (ru) -- (rd) -- cycle ;

    \end{tikzpicture}
    \caption{Шкала $\kframe{F}_m$}
    \label{fig1}
  \end{figure}

            Сначала заметим, что для того, чтобы формула $\alpha_m$ опровергалась
            в некотором мире (транзитивной) модели, достаточно, чтобы
            из этого мира была достижима шкала, изображённая на рис.~\ref{fig1}
            (чёрные кружк\'и соответствуют иррефлексивным мирам, светлые~---
            рефлексивным, отношение достижимости транзитивно). Обозначим
            обведённую на рис.~\ref{fig1} шкалу через $\kframe{F}_m^{\phantom{l}}$, а всю
            шкалу, изображённую на рис.~\ref{fig1},~--- через $\kframe{F}^+_m$.
            Формальное описание шкал $\kframe{F}_m^{\phantom{l}}$ и $\kframe{F}^+_m$ таково:
            $\kframe{F}_m^{\phantom{l}}=\opair{W_m^{\phantom{l}}}{R_m^{\phantom{l}}}$, $\kframe{F}^+_m=\opair{W^+_m}{R^+_m}$, где
            $$
            \begin{array}{rcl}
            W_m & = & \set{b^m,a^m_0,a^m_1,\dots,a^m_m}; \\
            W^+_m & = & W_m\cup\set{c^m},
            \end{array}
            $$
            а $R_m^{\phantom{l}}$ и $R^+_m$~--- транзитивные замыкания отношений
            $$
            \set{\opair{a^m_0}{b^m}, \opair{b^m}{b^m}} \cup
            \setc{\opair{a^m_i}{a^m_{i+1}}}{0\leqslant i\leqslant m-1}
            $$
            и
            $$
            R_m\cup \set{\opair{c^m}{c^m}, \opair{c^m}{a^m_0}}
            $$
            соответственно.

            Ясно, что если $\alpha_m$ истинна в некотором мире транзитивной
            модели, то она также истинна во всех мирах, достижимых из данного
            (поскольку главной связкой формулы $\alpha_m$ является $\Box$). Заметим
            также, что если $k \ne m$, то $\kframe{F}^+_m\models\alpha_k$.
            Используя эти наблюдения, мы построим требуемую модель $\kmodel{M}'$.
            Именно, чтобы получить $\kmodel{M}'$, расширим шкалу $\kframe{F}$, сделав
            достижимой копию шкалы $\kframe{F}_m$ из каждого мира множества $W$, в
            котором в $\kmodel{M}$ опровергается переменная $p_m$.

            Опишем модель $\kmodel{M}'$ формально.

            Для каждого $w\in W$ обозначим через $\kframe{F}_m^w$ копию шкалы
            $\kframe{F}_m$, помеченную миром $w$: положим $\kframe{F}_m^w=
            \langle W_m^w, R_m^w\rangle$, где $W^w_m= W_m\times\{w\}$ и
            для всяких $x,y\in W_m$
            $$
            \begin{array}{rcl}
            \otuple{x,w} R_m^w \otuple{y,w}
            & \bydef &
            xR_my.
            \end{array}
            $$
            Положим
            $$
            \begin{array}{rcl}
            W' & = & \displaystyle W \cup \bigcup\setc{W_m^w}{\mbox{$1\leqslant
            m\leqslant 2n + 2$, $w\in W$, $(\kmodel{M},w) \not\models p_m$}}.
            \end{array}
            $$
            На множестве $W'$ определим отношение $\tilde R$:
            $$
            \begin{array}{rcl}
            w\tilde{R}w'
              & \bydef
              & \mbox{либо $w,w'\in W$ и $wRw'$,}
            \\
              &
              & \mbox{либо $w,w'\in W_m^u$ и $wR_m^uw'$,}
            \\
              &
              & \mbox{либо $w\in W$, $(\kmodel{M},w) \not\models p_k$
                 и $w'= \otuple{a_0^m,w}$}.
            \end{array}
            $$
            Обозначим через $R'$ транзитивное замыкание отношения $\tilde R$.
            Пусть $\kframe{F}'=\otuple{W',R}$ и пусть $v'$~---
            произвольная оценка пропозициональных переменных в мирах из множества
            $W'$. Положим $\kmodel{M}'=\otuple{\kframe{F}',v'}$.

            Заметим, что для каждых $w\in W'$ и $m\in\set{1,\ldots,n+1}$
            $$
            \begin{array}{rcl}
            (\kmodel{M}', w) \not\models \alpha_m
            & \iff &
            \mbox{$w\in W$ и $(\kmodel{M}, w) \not\models p_m$.}
            \end{array}
            \eqno{\mbox{$(\ast)$}}
            $$
            Действительно, для опровержения $\alpha_m$ в $w$ миру $w$ требуется видеть мир, из которого можно попасть в слепой мир за $m$ шагов, но не больше (а значит, он иррефлексивен), а также можно видеть бесконечную цепь миров (например, рефлексивный мир); это возможно, только если мы находимся в корне копии шкалы $\kframe{F}_m$, а значит, эта копия $\tilde R$\nobreakdash-достижима из $w$ или мира, $R$\nobreakdash-достижимого из~$w$, причём находящегося в~$W$. Учитывая условие наследственности <<вверх>> для оценки в $\kmodel{M}$, получаем, что $(\kmodel{M}, w) \not\models p_i$ по построению модели~$\kmodel{M}'$. Если же $w\in W$ и $(\kmodel{M}, w) \not\models p_i$, то $(\kmodel{M}, w) \not\models \alpha_i$ по построению~$\kmodel{M}'$.

            Для всякой формулы $\psi$ от переменных
            $p_1,\dots,p_{2n+2}$ обозначим через $\psi_\alpha$ формулу,
            полученную из $\psi$ подстановкой $\alpha_1, \dots , \alpha_{n+2}$
            вместо $p_1, \dots , p_{n+2}$ соответственно. Тогда для всякой
            подформулы $\psi$ бескванторной булевой формулы $\varphi'$
            и всякого мира $w \in W$ уровня $n$ имеет место следующая
            эквивалентность:
            $$
            \begin{array}{rcl}
            (\kmodel{M}', w) \models \psi_\alpha
            & \iff &
            (\kmodel{M}, w) \models \psi.
            \end{array}
            $$
            Справедливость данного обосновывается индукцией по построению~$\psi$:
            случай, когда $\psi = \bot$, тривиален; если $\psi = p_m$,
            то $\psi_\alpha = \alpha_m$, и требуемое получаем по~\mbox{$(\ast)$};
            индукционный шаг также тривиален (напомним, что $\psi$~--- безмодальная
            пропозициональная формула).

            Таким образом, если $w$~--- мир модели $\kmodel{M}$ уровня $n$, то
            $(\kmodel{M}',w) \models \varphi'_\alpha$. Заметим, что, согласно~\mbox{$(\ast)$}, формула
            $\alpha_{2n+1} \wedge\neg \alpha_{2n+2}$ истинна в некотором мире
            $w$ модели $\kmodel{M}'$ только если $w$ является миром уровня $n$ в
            $\kmodel{M}$, поскольку миры модели $\kmodel{M}$, в которых истинна
            формула $q_n \wedge\neg q_{n+1}$,~--- это в точности миры
            уровня~$n$. Следовательно, $(\kmodel{M}',w_0) \models
            \Box^+(\alpha_{2n+1}\wedge\neg\alpha_{2n+2}\to\varphi'_\alpha)$.

            Осталось проверить, что
            $$
            (\kmodel{M}', w_0) \models \alpha_{n+1} \wedge \neg \alpha_{n+2} \wedge
            \Box^+ A_\alpha \wedge
            \Box^+ B_\alpha \wedge
            \Box^+ C_\alpha \wedge
            \Box^+ D_\alpha \wedge
            \Box^+ E_\alpha.
            $$

            Так как $(\kmodel{M},w_0)\models q_0$, то, согласно~\mbox{$(\ast)$}, получаем $(\kmodel{M}',w_0)\models\alpha_{n+1}$ (напомним, что
            $p_{n+1}= q_0$).
            Поскольку $(\kmodel{M},w_0) \not\models q_1$, то, снова по~\mbox{$(\ast)$},
            формула $(\kmodel{M}',w_0)\models\neg\alpha_{n+2}$.

            Предположим, что $(\kmodel{M}',w_0)\not\models\Box^+ A_\alpha$.
            Тогда для некоторого $w\in W'$, достижимого из $w_0$, имеет место
            отношение $(\kmodel{M}',w)\not\models A_\alpha$, т.\,е.
            $$
            \begin{array}{c}
            \displaystyle
            (\kmodel{M}',w)\not\models
            \bigwedge\limits_{i=n+2}^{2n+2}(\alpha_i\to\alpha_{i-1}).
            \end{array}
            $$
            Заметим, что $w\in W$. Действительно, во всяком мире шкалы
            $\kframe{F}_k$ истинны все формулы $\alpha_i$, поэтому истинна и
            конъюнкция их импликаций друг к другу. Итак, $w\in W$, причём для
            некоторого $i\in\set{n+2,\dots,2n+2}$ имеет место
            отношение $(\kmodel{M}',w)\not\models\alpha_i\to\alpha_{i-1}$.
            Но тогда $(\kmodel{M}',w)\models\alpha_i$ и
            $(\kmodel{M}',w)\not\models\alpha_{i-1}$, а следовательно,
            $(\kmodel{M},w)\models p_i$ и
            $(\kmodel{M},w)\not\models p_{i-1}$, из чего несложно заключить,
            что $(\kmodel{M},w)\not\models A$, т.\,е.
            $(\kmodel{M},w_0)\not\models\Box^+A$. Получили противоречие,
            следовательно, $(\kmodel{M}',w_0)\models\Box^+ A_\alpha$.

            Предположим, что $(\kmodel{M}',w_0)\not\models\Box^+ B_\alpha$.
            Тогда для некоторого $w\in W'$, достижимого из $w_0$, имеет место
            отношение $(\kmodel{M}',w)\not\models B_\alpha$.
            %
            %
            Это означает, что для некоторого $i\in\set{1,\dots,n}$
            $$
            \begin{array}{l}
            (\kmodel{M}',w)\models\alpha_{n+1+i},
            \\
            (\kmodel{M}',w)\not\models
            \big(\alpha_i\to\Box(\alpha_{n+1+i}\to
            \alpha_i)\big) \wedge
            \big(\neg \alpha_i\to\Box(\alpha_{n+1+i}\wedge\neg
            \alpha_{2n+2}\to\neg \alpha_i)\big).
            \end{array}
            $$

            Пусть $(\kmodel{M}',w) \not\models \alpha_i \to
            \Box(\alpha_{n+1+i} \to \alpha_i)$. Тогда
            $(\kmodel{M}',w)\models\alpha_i$ и существует $w'\in W'$ такой, что
            $wR'w'$, $(\kmodel{M}',w') \models \alpha_{n+1+i}$,
            $(\kmodel{M}',w') \not\models \alpha_i$. Так как
            $(\kmodel{M}',w') \not\models \alpha_i$, то $w'\in W$, а
            значит, и $w\in W$, откуда следует, что
            $$
            \begin{array}{rcl}
            (\kmodel{M},w) \not\models
            q_i\to\big(p_i\to\Box(q_i\to
            p_i)\big)\wedge\big(\neg p_i\to\Box(q_i\wedge\neg
            q_{n+1}\to\neg p_i)\big),
            \end{array}
            $$
            а значит, $(\kmodel{M},w_0) \not\models\Box^+B$. Получили
            противоречие.

            Пусть $(\kmodel{M}',w) \not\models
            \neg \alpha_i\to\Box(\alpha_{n+1+i}\wedge\neg
            \alpha_{2n+2}\to\neg \alpha_i)$. Тогда
            $(\kmodel{M}',w) \not\models \alpha_i$ и существует $w'\in W'$
            такой, что $wR'w'$, $(\kmodel{M}',w')\models\alpha_{n+1+i}\wedge\neg
            \alpha_{2n+2}$,
            \linebreak[4]
            $(\kmodel{M}',w')\models\alpha_i$. Так как
            $(\kmodel{M}',w)\not\models \alpha_i$, то $w\in W$, а так как
            $(\kmodel{M}',w') \not\models \alpha_{2n+2}$, то $w'\in W$.
            Следовательно, $(\kmodel{M},w)\not\models p_i$,
            $(\kmodel{M},w)\models q_i$,
            $(\kmodel{M},w')\models q_i$,
            \linebreak
            $(\kmodel{M},w')\not\models q_{n+1}$,
            $(\kmodel{M},w')\models p_i$. Но тогда
            $$
            \begin{array}{rcl}
            (\kmodel{M},w) \not\models
            q_i\to\big(p_i\to\Box(q_i\to
            p_i)\big)\wedge\big(\neg p_i\to\Box(q_i\wedge\neg
            q_{n+1}\to\neg p_i)\big),
            \end{array}
            $$
            а значит, $(\kmodel{M},w_0) \not\models\Box^+B$. Снова получили
            противоречие.

            Значит, $(\kmodel{M}',w_0)\models\Box^+ B_\alpha$.

            Предположим, что $(\kmodel{M}',w_0) \not\models\Box^+ C_\alpha$.
            В этом случае для некоторого
            \linebreak
            $w\in W'$,
            достижимого из $w_0$, имеет
            место отношение $(\kmodel{M}',w)\not\models C_\alpha$.  Тогда для
            некоторого $i\in\set{0,\dots,n-1}$ такого, что $Q_{i+1}=
            \forall$, выполняются отношения $(\kmodel{M}',w)\models
            \alpha_{n+1+i}\wedge\neg\alpha_{n+2+i}$ и
            $(\kmodel{M}',w)\not\models
            \Diamond(\alpha_{n+2+i} \wedge \neg\alpha_{n+3+i} \wedge
            \alpha_{i+1})$.
            Так как $(\kmodel{M}',w)\not\models\alpha_{n+2+i}$, то $w\in
            W$, а так как $(\kmodel{M}',w)\models
            \alpha_{n+1+i}\wedge\neg\alpha_{n+2+i}$, то
            $(\kmodel{M},w)\models
            q_{i}\wedge\neg q_{i+1}$. Поскольку
            $(\kmodel{M},w_0)\models\Box^+C$, получаем, что
            $(\kmodel{M},w)\models
            q_{i}\wedge\neg q_{i+1} \to \Diamond(q_{i+1}\wedge\neg
            q_{i+2}\wedge p_i)$, а следовательно, существует
            мир $w'\in W$ такой,
            что $wRw'$ и $(\kmodel{M},w')\models
            q_{i+1}\wedge\neg q_{i+2}\wedge p_i$.
            По
            построению
            модели
            $\kmodel{M}'$ в этом случае имеет место отношение
            $(\kmodel{M}',w')\models
            \alpha_{n+2+i} \wedge \neg\alpha_{n+3+i} \wedge
            \alpha_{i+1}$. Так как $wRw'$, то $wR'w'$, а значит,
            $(\kmodel{M}',w)\models
            \Diamond(\alpha_{n+2+i} \wedge \neg\alpha_{n+3+i} \wedge
            \alpha_{i+1})$. Получили противоречие, следовательно,
            $(\kmodel{M}',w_0)\models\Box^+ C_\alpha$.

            Истинность в мире $w_0$ модели $\kmodel{M}'$ формул $\Box^+D_\alpha$ и
            $\Box^+E_\alpha$ доказывается аналогично тому, как это было сделано в
            случае формулы~$\Box^+C_\alpha$.

            В результате получаем, что $(\kmodel{M}', w_0) \models
            \varphi^\ast_\alpha$. Осталось заметить, что $\kframe{F}'$~--- шкала
            логики $\logic{wGrz}$, и значит,
            $\varphi^\ast_\alpha$ является $\logic{wGrz}$\nobreakdash-выполнимой.
            \end{proof}

            В качестве следствия сразу же получаем следующее утверждение.

            \begin{theorem}
            \label{th_PSPACE_K4}
            \label{th:PSPACE:K4}
            Проблема разрешения константного фрагмента логики $L\in [\logic{K},\logic{wGrz}]$ является\/ $\cclass{PSPACE}$\nobreakdash-трудной.
            \end{theorem}

            \begin{proof}
            Достаточно показать, что $\logic{TQBF}$ полиномиально сводится к проблеме $L$\nobreakdash-выполнимости константных формул. Без ограничений общности можем рассматривать только формулы вида $\neg Q_1p_1\ldots Q_np_n\varphi'$, где $\varphi'$~--- бескванторная формула от переменных $p_1,\ldots,p_n$. Действительно, булеву формулу с кванторами можно привести к префиксной нормальной форме, а затем взять её двойное отрицание, оставить внешнее отрицание, а внутреннее пронести через кванторную приставку.

            Пусть $\varphi = Q_1p_1\ldots Q_np_n\varphi'$. Покажем, что
            $$
            \begin{array}{rcl}
            \neg\varphi\in\logic{TQBF}
              & \iff
              & \neg\varphi^\ast_\alpha\in\logic{L}.
            \end{array}
            $$

            Пусть $\neg\varphi\in\logic{TQBF}$. Тогда, согласно предложению~\ref{prop:modal:complexity}, $\neg\varphi^\ast\in\logic{K}$. Логика $\logic{K}$ замкнута по правилу подстановки, поэтому $\neg\varphi^\ast_\alpha\in\logic{K}$. Поскольку $\logic{K}\subseteq L$, получаем, что $\neg\varphi^\ast_\alpha\in L$.

            Пусть $\neg\varphi\not\in\logic{TQBF}$. Тогда, согласно предложению~\ref{prop:modal:complexity}, $\neg\varphi^\ast\not\in\logic{wGrz}$. В этом случае, согласно лемме~\ref{lem_K4(0)form_equiv}, $\neg\varphi^\ast_\alpha\not\in\logic{wGrz}$. Поскольку $L\subseteq\logic{wGrz}$, получаем, что $\neg\varphi^\ast_\alpha\not\in L$.

            Осталось заметить, что функция, сопоставляющая формуле $\varphi$ формулу $\varphi^\ast_\alpha$, является полиномиально вычислимой.
            \end{proof}


\newcommand{\nosort}[1]{}\newcommand{\titlefont}{\rm}

\end{document}